\documentclass[11pt]{amsart} \usepackage[latin1]{inputenc}
\usepackage{amsmath,amsthm, amscd, amssymb, amsfonts}
\usepackage{enumerate}
\usepackage{hyperref}

\newtheorem{teor}{Theorem}[section] \newtheorem{corol}[teor]{Corollary} \newtheorem{prop}[teor]{Proposition} \newtheorem{lem}[teor]{Lemma}

\theoremstyle{definition} \newtheorem{defin}[teor]{Definition}  \newtheorem{ejem}[teor]{Example}

\theoremstyle{remark} \newtheorem{obs}[teor]{Remark}\newtheorem{nota}[teor]{Remark}

\DeclareMathOperator*{\sedi}{\rtimes}
\DeclareMathOperator*{\sset}{\subseteq}
\DeclareMathOperator*{\accion}{\rightharpoonup}
\DeclareMathOperator*{\acciond}{\leftharpoonup}

\newcommand\M{\mathcal{M}}

\newcommand\g{\gamma}
\newcommand\s{\sigma}\newcommand \End{\operatorname{End}}
  
\newcommand\co{\operatorname{co}}  
  
 \newcommand\Hom{\operatorname{Hom}} 
 \newcommand\Ind{\operatorname{Ind}}\newcommand\Aut{\operatorname{Aut}}
\newcommand\Gal{\operatorname{Gal}}

\newcommand\rad{\operatorname{rad}}
\newcommand\can{\operatorname{can}}

\begin{document}

\title[On the classification of Galois objects for finite groups]{On the classification of Galois objects for finite groups} \author{C\'esar Galindo and Manuel Medina} \address{Departamento de Matem\'aticas, Pontificia Universidad Javeriana, Bogot\'a, Colombia} \email{cesar.galindo@javeriana.edu.co, mjmedinal@bt.unal.edu.co} \thanks{This work was
partially supported by the Project 003568 of the Department of Mathematics, Pontificia Universidad Javeriana}  \date{June 18, 2010}

\begin{abstract}
We classify Galois objects for the dual of a group algebra of a finite group over an arbitrary field.
\end{abstract}

\maketitle \section{Introduction}

Let $H$ be a Hopf algebra. A right $H$-comodule algebra $A$ is called a right $H$-Galois object if the map can:$A\otimes A\to A\otimes H, a\otimes b\mapsto ab_{(0)}\otimes b_{(1)}$ is bijective. When the Hopf algebra $H$ is the dual of a group algebra of a finite group, then Galois objects are the $G$-Galois extensions of non-commutative rings introduced by Chase, Harrison and Rosenberg \cite{chaseSwiddler}.
\medbreak
The Galois objects are very interesting because have a categorial interpretation, and they can be used for the construction of new Hopf algebras. Let $^H\M$ be the category of finite dimensional left $H$-comodules. A fiber functor $F: ^H\M\to $Vec$_k$ is an exact and faithful monoidal functor that commutes with arbitrary colimits. Ulbrich defined in \cite{Ulbrich} a fiber functor $F_A$ associated with each $H$-Galois object $A$, in the form $F_A(V ) = A\square_HV$, where $A\square_HV$ is the cotensor product over $H$ of the right $H$-comodule $A$ and the left $H$-comodule $V$. In \emph{loc. cit.} was defined a bijective correspondence between isomorphism classes of $H$-Galois objects and isomorphism classes of fiber functors of $^H\M$.
\medbreak
Given a Hopf algebra $H$ and a left $H$-Galois object $A$, there is a new Hopf algebra associated $L(A,H)$, see \cite{Scha96}. The Hopf algebra $L(A,H)$ is the Tannakian-Krein reconstruction from the fiber functor associated to $A$, \cite[Theorem 5.5]{Scha96}.
\medbreak

It is well known that the classification of $H^*$-Galois objects for a finite dimensional Hopf algebra $H$, is equivalent to the classification of twist in $H$, see
\cite[Section 4]{Twisting}. In \cite{Mov}, Movshev classified the twist for complex group algebras, so this implies the classification of Galois objects for the dual of a complex group algebra of a finite group.
\medbreak
Throughout this article we work over an arbitrary field $k$, and we will denote by $k^G$ the dual of the group algebra for a finite group $G$.
\medbreak

The aim of this paper is to classify $k^G$-Galois objects, generalizing the classification of Movshev \cite{Mov}, and Davydov \cite{Davy}.
\medbreak

To formulate our main result we first describe a Galois datum associated to a finite group $G$ and a field $k$.

\begin{defin}\label{defini intro}
A\emph{Galois datum associated to $k^G$} is a collection $(S,K,N,\s,\gamma)$ such that
\begin{enumerate}[$i)$]
\item $S$ is a subgroup of $G$ and $N$ is a normal subgroup of $S$.
\item $K\supseteq k$ is a Galois extension with Galois group $S/N$.
\item char$(k)\nshortmid |N|$.
\item $\s:N\times N\to K^*$ is a non-degenerate 2-cocycle.
\item $\gamma:S\times N\to K^*$  satisfies the equations \eqref{condicion 1 gamma}, \eqref{condicion 2 gamma}, and \eqref{condicion 3 gamma}.
\end{enumerate}
\end{defin}

A 2-cocycle $\sigma\in Z^2(N,K^*)$ is called non-degenerate if $\s(s, t) = \s(t, s)$ for all $t \in C_N(s)$ implies $s=1$. Equivalently,  $\s$ is non-degenerate if the center of the  twisted group algebra  $K_\sigma N$ coincides with $K$. Also, by Theorem of Maschke for twisted group algebras \cite[Theorem 2.10, pag 85]{karpi}, if char$(k)\nshortmid |N|$ and $\sigma$ is non-degenerate then $K_\sigma N$ is a central simple algebra over $K$.
\medbreak
The function $\gamma$ is univocally determined by  a Hochschild 1-cocycle $\gamma\in HZ^1(S, C^1(N,K^*))$, where $C^1(N,K^*)$ is an $S$-bimodule with actions given by the equations \eqref{ecuacion accion de C^1 a derecha}, and \eqref{ecuacion accion de C^1 a izquierda}, see Section \ref{seccion obstruccion}. In the Section  \ref{seccion obstruccion}  we define some cohomological obstructions which establish necessary conditions for the existence of $\gamma$.
\medbreak
Let $(S,K,N,\s,\gamma)$ be a Galois datum associated to $k^G$. The group $S$ acts over $K$ through the canonical projection to $\text{Gal}(K|k)=S/N$. We will denote by $A(K_\sigma N,\gamma)$ the twisted group algebra $K_\sigma N$ with $S$-action defined by
\begin{equation*}
g\accion \alpha u_x= \bar{g}(\alpha)\gamma(g,x)u_{\-^gx},
\end{equation*}
for $g\in S$, $x\in N$, and $\alpha\in K$.
\medbreak
In the Subsection \ref{subsection algebra inducida}, we study the induction functor from the category of $S$-algebra to the category of $G$-algebras. Using the induction functor we define the $G$-algebra $\Ind_S^G(A(K_\sigma N,\gamma))$.
\medbreak
Now we can formulate our main result.

\begin{teor}\label{main result}
Let $G$ be a finite group and let $k$ be a field.

\begin{enumerate}[i)]
  \item Let $(S,K,N,\s,\gamma)$ be a Galois datum associated to $k^G$. Then the $G$-algebra $\Ind_S^G(A(K_\sigma N,\gamma))$ is a $k^G$-Galois object.
  \item Let $A$ be a $k^G$-Galois object. Then $A\simeq \Ind_S^G(A(K_\sigma N,\gamma))$ for a Galois datum $(S,K,N,\s,\gamma)$.
\end{enumerate}
\end{teor}

In the Subsection \ref{subseccion datos equivalentes} we establish an equivalence among the Galois data, this equivalent data define isomorphic Galois objects.

\medbreak
The paper is organized as follows: in Section \ref{prels} we give a brief account of
results on Hopf-Galois extensions that will be needed in the sequel. In Section \ref{groups} we show that every $k^G$-Galois object is the induced of a simple Galois object. For the proof we define the concept imprimitive system for $G$-algebras. In Section \ref{seccion 4} we show that every simple $k^G$-Galois object is a twisted group algebra. For this proof we use as mean tool the Miyashita-Ulbrich action. In Section \ref{seccion 5} we define $k^G$-Galois objects associates to some group-theoretical data. In Section \ref{seccion 6} we proof our main result, see Theorem \ref{main result}. Also, we give some non-trivial examples of Galois objects. In Section \ref{seccion obstruccion} we present some cohomological obstruction for the existence of a function $\gamma$, see Definition \ref{defini intro}.

%\subsection*{Acknowledgement} The authors thank xxxxxxx for his interest and valuable remarks on this paper.

\section{Preliminaries on Hopf Galois extensions}\label{prels}

In this section we review some results on Hopf Galois extensions that we will need later. We refer the reader to \cite{galoissurvey} for a
detailed exposition on the subject.

\begin{defin}\label{defin hopf galois} Let $H$ be a Hopf algebra. Let also $A$ be a right $H$-comodule algebra with structure map $\rho: A\to
A\otimes H$, $\rho(a) = a_{(0)}\otimes a_{(1)}$, and let $B= A^{\co H}$.  The extension $A\supseteq B$ is called a right \emph{Hopf Galois}
extension, or a right $H$-\emph{Galois}  extension if the canonical map $$\can: A\otimes_BA\to A\otimes H, \  \  a\otimes b\mapsto
ab_{(0)}\otimes b_{(1)},$$is bijective.

A right $H$-Galois extension of the base field $k$ will be called a right $H$-\emph{Galois object}. Left $H$-Galois extensions and left
$H$-Galois objects are defined similarly. \end{defin}

\begin{ejem}\label{ejem2}Let the Hopf algebra $H$ coact on itself through the comultiplication $\Delta$. Then the canonical map $H\otimes H \to
H\otimes H$, $x\otimes y\mapsto  xy_{(1)}\otimes y_{(2)}$ is bijective with inverse $x\otimes  y \mapsto  xS(y_{(1)})\otimes y_{(2)}.$
\end{ejem}

\emph{Assume from now on that $H$ is finite dimensional Hopf algebra.} \medbreak

We have a characterization of  $H^*$-Galois objects $A$ in terms of the natural
$H$-action over $A$:

\begin{prop}\label{prop equivalencia galois}
$A$ is a right $H^*$-Galois object if and only if $A$ is finite dimensional and the map $\theta: A\#H\to \End A$, $\theta(a \# h)(b) = a(h\cdot b)$ is an isomorphism.
\end{prop} \begin{proof}
See \cite[8.3.3 Theorem]{Mont}.
\end{proof}

\begin{prop}\label{semisimple galois}
Every right $H$-Galois object is isomorphic to a crossed product $H^{\sigma} = k\#_{\sigma}H$, where
$\sigma:H\otimes H\to k$ is an invertible 2-cocycle. Moreover, if $H$ is semisimple then  $H^{\sigma}$ is semisimple.
\end{prop}

On the other hand, $H^{\sigma}$ is a right $H$-Galois object for all such $\sigma$.

\begin{proof}
Let $A$ be a right $H$-Galois object. Since $A \# H^* \simeq \End A$ is simple artinian, thus by  \cite[Proposition 8.3.6]{Mont} there is $\sigma: H\otimes H\to k$ an invertible 2-cocycle such that $$A\simeq k\#_{\sigma}H.$$
 Now by \cite[Theorem 7.4.2]{Mont}, if $H$ is semisimple and
finite dimensional then $A\simeq k\#_{\sigma}H$ is semisimple.
\end{proof}

\begin{obs}
It follows by the Proposition \ref{semisimple galois}, that for an $H$-Galois object $A$, $\dim_k(A)=\dim_k H$.
\end{obs}

\subsection{Miyashita-Ulbrich action}\label{miy-ul}
Let $A \supseteq B$ be an $H$-Galois extension. The Miyashita-Ulbrich action of $H$ on the centralizer $A^B$ of $B$ in $A$, makes $A^B$ into a commutative
algebra in $\mathcal{YD}_H^H$ the category of right Yetter-Drinfeld modules.

\begin{defin} Let $H$ be a Hopf algebra, and $A$ an $H$-Galois extension of $B$. The Miyashita-Ulbrich action of $H$ on $A^B$ is defined by
$a\triangleleft h = h^{[1]}ah^{[2]}$, $a\in A^B$, $h\in H$, where $\can(h^{[1]}\otimes h^{[2]}) = 1\otimes h$. \end{defin}

The Miyashita-Ulbrich action of $H$ on the $H$-Galois object $A$ is characterized as the unique map $A\otimes H\to A$, $a\otimes h\mapsto
a\triangleleft h$, such that $ab=b_{(0)}(a\triangleleft b_{(1)})$, for all $a, b \in A$.

\begin{ejem}
Consider the $H$-Galois object $A=H$ as in Example \ref{ejem2}. In this case, the Miyashita-Ulbrich action coincides with the right adjoint
action of $H$ on itself: $a\triangleleft h = \mathcal S(h_{(1)})ah_{(2)}$, $a, h \in H$.

Therefore a Hopf subalgebra $H'\subseteq  H$ is a normal Hopf subalgebra if and only if it is stable
under the Miyashita-Ulbrich action.
\end{ejem}

The next theorem will be useful to reduce the classification of $H$-Galois object  in the case that we have an exact succession $k\to K\to H\to Q\to k$ of Hopf algebras.

\begin{teor}\label{teor reduccion sucesion}
Let $H$ be a finite dimensional Hopf algebra, $K\subseteq H$ a normal Hopf subalgebra, and $Q=H/K^+H$ the quotient Hopf algebra. Let $\delta:A\to A\otimes H$ be an $H$-comodule algebra. Then the following are equivalent:
\begin{enumerate}[i)]
  \item $A$ is an $H$-Galois objet.
  \item $A^{co Q}$ is a $K$-Galois object, and $A^{co Q}\subseteq A$ is a faithfully flat $Q$-Galois extension.
\end{enumerate}
\end{teor}
\begin{proof}
See \cite[Theorem 4.5.1]{galoissurvey}
\end{proof}
In the case that $H=k^G$ for a finite group $G$, and $A$ is a simple algebra we have:

\begin{corol}\label{prop teorema takeuchi importante}
Let $A$ be a $k^G$-comodule algebra, where $A$ is a simple algebra with center $K$. Then $A$ is a $k^G$-Galois object if and only if $A$ is a $K^N$-Galois object  $(\text{where} \ N=\{x\in G| x\cdot \alpha=\alpha, \forall \alpha\in K\})$ and $K$ is a $k^{G/N}$-Galois object.
\end{corol}
\begin{proof}
Suppose that $A$ is a $k^G$-Galois object. If we show that $A^N=K$, the necessity follows by Theorem \ref{teor reduccion sucesion}.

Since $G$ acts over $A$ and $K=\mathcal{Z}(A)$, then we have a homomorphism $\pi: G\to \Aut_k(K)$, such that $\ker\pi= N$. Thus $\pi$ induces an injective homomorphism  $\widetilde{\pi}: G/N\to \Aut_k(K)$. Now, $A^G=k$ implies $K^{G/N}=k$. Hence by \cite[Theorem 13]{Artin} we have $[K:k]\geq |G/N|$.

By Theorem \ref{teor reduccion sucesion} $A^N$ is a $k^{G/N}$-Galois object, so $\text{dim}_kA^N= |G/N|$. Since $K\subseteq A^N$, and $[K:k]\geq |G/N|$, thus $A^N=K$.

The Sufficiency follows by Theorem \ref{teor reduccion sucesion}.

\end{proof}
\section{Galois extensions of finite groups}\label{groups}

Let $G$ be a finite group, and let $A$ be a $G$-algebra with $G$-action $kG \otimes A \to A$, $g \otimes a \mapsto g\accion a$.
Equivalently, $A$ is a $k^G$-comodule algebra with respect to the coaction $\rho: A \to A \otimes k^G$ defined by $$\rho(a) = \sum_{g \in G} g\accion a
\otimes t_g,$$ where $(t_g)_{g \in G}$ denotes the basis of $k^G$ consisting of the canonical idempotents $t_g(h) =
\delta_{g, h}$ (Kronecker's delta).

The next proposition is a restatement of \cite[Proposition 3.1]{Davy} when $k$ is not an algebraically closed field.
\begin{prop}\label{caracterizacion A.Ga=ideales G-inv}
Let $A$ be a $G$-algebra. Then $A$ is a $k^G$-Galois object if and only if $A$ satisfies the following conditions:
\begin{enumerate}[i)]
  \item $\dim(A)=|G|$. %Dimension of $A$ is equal to order of $G$.
  \item $A$ has no no-trivial $G$-invariant left ideals.
  \item $A^G=k$.
\end{enumerate}
\end{prop}

\begin{proof}
Let $J$ be a $G$-invariant left ideal of $A$. By \cite[Theorem 2.3.9]{galoissurvey} the map $A\otimes (A/J)^G\to A/J,\ a\otimes m\mapsto a\cdot m$ is an isomorphism. Now,  if $J\neq 0$, then $A/J=0$, so $A=J$.\\
Conversely, we can considerer $A$ as a left $A\sedi G$-module with structure map $\theta$. By $(ii)$, $A$ is a simple $A\sedi G$-module. Let $D=\End_{A\sedi G}(A)$. Applying the Jacobson Density Theorem (see \cite[F20, pag 139]{Lorenz}), we have that natural homomorphism
\[
\theta:A\sedi G\to \End_{D}(A)=\End_{A^G}(A)=\End_k(A)
\]
is surjective, then by $(i)$ $\theta$ is bijective.
\end{proof}

\begin{nota}\label{1 diferencia con Davydov}
The condition $A^G=k$ in Proposition \ref{caracterizacion A.Ga=ideales G-inv} is not necessary if $k$ is algebraically closed.
\end{nota}
The following corollary is an alternative proof of the Proposition \ref{semisimple galois} for the case $H=k^G$.
\begin{corol}\label{algebra de galois semisimple}
Every $k^G$-Galois object is a semisimple algebra.
\end{corol}

\begin{proof}
Let $A$ be a $k^G$-Galois object. Since the Jacobson radical $\rad(A)$ of the algebra $A$ is a left $G$-invariant ideal, we conclude that $\rad(A)=0.$
\end{proof}

\subsection{Imprimitive algebras and induced algebras}\label{subsection algebra inducida}

In this subsection we introduce the notion of imprimitive algebra.

\begin{defin}
Let $A$ be a $G$-algebra and $\{e_1,\ldots, e_n\}$ a set of orthogonal central idempotents in $A$, such that $1_A= e_1+\cdots + e_n$. We will say that $\{e_1,\ldots, e_n\}$ is an \emph{imprimitive system} if  $g\accion e_i\in \{e_1,\ldots, e_n\}$ for all $g\in G, i=1,\ldots,n$.

A $G$-algebra is called \emph{imprimitive} if there is a non-trivial imprimitive system in $A$.
\end{defin}

Let $S$ be a subgroup of $G$ and $(B,\cdot)$ an $S$-algebra. The vector space $kG\otimes_{kS}B$ with $G$-action and multiplication defined by:
\begin{align*}
    g\accion (h\otimes x)&= gh\otimes x\\
    (g\otimes x)(h\otimes y)&=\begin{cases}
    h\otimes (h^{-1}g\cdot x)y &\text{ if } h^{-1}g\in S\\
    0  & \text{ if } h^{-1}g\notin S,
    \end{cases}
\end{align*}
is a $G$-algebra for $g,h\in G$ and $x,y\in B$

Also, considerer the algebra of functions
\[
A_S(G,B)=\{r:G\to B|r(sg)=s\cdot r(g) \quad \forall s\in S,g\in G\},
\]
which is a $G$-algebra with action defined by $(g\accion r)(x)=r(xg)$.

\begin{lem}
$A_S(G,B)\simeq kG\otimes_{kS}B$ as $G$-algebras.
\end{lem}
\begin{proof}
  See \cite[Proposition 3.3]{GN2}.
\end{proof}

\begin{defin}\label{defin algebra inducida}
 Let us denote by $\Ind_S^G(B)$ the $G$-algebra $A_S(G,B)\simeq kG\otimes_{kS}B$ and we will call the \emph{induced algebra} from the $S$-algebra $B$.
\end{defin}

\begin{nota}\label{nota ind is a functor}
 Induction is a covariant functor: where to each homomorphism
of $S$-algebras $f:A\to B$ is send to the homomorphism of $G$-algebras $\Ind_S^G(f):\Ind_S^G(A)\to \Ind_S^G(B)$, $\Ind_S^G(f)(r)=f\circ r$.
\end{nota}

\begin{prop}\label{caracterizacion algebras imprimitivas}
Let $A$ be an imprimitive algebra, such that $G$ acts transitively on the set $\{e_1,\ldots, e_n\}$. If $S$ is the stabilizer of $e_1$ then $A\simeq \Ind_S^G(e_1A)$ as $G$-algebras.
\end{prop}

\begin{proof}
Let $\{g_1,\ldots,g_n\}$ be a set of representatives for the right cosets $G/S$ such that $g_i\accion e_1=e_i$. Then $\Ind_S^G(e_1A)=kG\otimes_{kS}e_1A = \bigoplus_{i=1}^n g_i\otimes e_1A.$

Considerer the map
\begin{align*}
    \psi:kG\otimes_{kS}e_1A&\to A\\
    \sum_{i=1}^ng_i\otimes e_1a_i&\mapsto \sum_{i=1}^ne_i(g_i\accion a_i).
\end{align*}
It is straightforward to check that the map $\psi$ is a well defined  $G$-algebra isomorphism.
\end{proof}

\begin{lem}\label{lema inducida=B y algebra invariantes}
Let $S$ be a subgroup of $G$, and $(B,\cdot)$ be an $S$-algebra. Then there is an algebra isomorphism between the induced algebra $\Ind_S^G(B)$ and $\underbrace{B\times\cdots\times B}_{[G:S]}$. Moreover,
$
(\Ind_S^G(B))^G\simeq B^S.
$
\end{lem}

\begin{proof}
  Let  $n=[G:S]$ and $\{1_G,g_2,\ldots,g_n\}$ be a set of representatives for the right cosets $G/S$. Considerer the functions $e_{i}\in\Ind_S^G(B)$ defined by
\begin{equation}\label{ecua sistema imprimitivo canonico}
e_{i}(x)=\begin{cases}
1_B, &\text{ si } x\in Sg_i,\\
0, &\text{ si } x\notin Sg_i.
\end{cases}
\end{equation}
The set $\{e_1,\ldots,e_n\}$ is an imprimitive system in $\Ind_S^G(B)$, thus
$$
\Ind_S^G(B)\simeq e_{1}\Ind_S^G(B)\times \cdots \times e_{n}\Ind_S^G(B).
$$
It is clear that $B\simeq e_i\Ind_S^G(B)$, so $\Ind_S^G(B)\simeq \underbrace{B\times\cdots\times B}_{[G:S]}$.

Now, for the last claim, the map
\begin{align*}\label{trans. B encaja Inducida}
    B^S&\to (\Ind_S^G(B))^G\\
    b&\mapsto [x\mapsto b],\notag
\end{align*}
is an algebra isomorphism.
\end{proof}

\begin{nota}\label{nota sistema imprimitivo canonico}
  We will say that the system $\{e_i\}$, defined in \eqref{ecua sistema imprimitivo canonico}, is the \emph{canonical imprimitive system associated to $\Ind_S^G(B)$}.
\end{nota}
\subsection{Galois Objects as Induced Algebras}\label{algebra de Galois asociada a un algebra inducida}

In this section we show that every Galois object of a finite group $G$ is isomorphic to the induced algebra of a simple Galois object of a subgroup of $G$ (see Theorem \ref{teor principal cap 2}).

\begin{lem}\label{lema E_i sistema imprimitivo}
Let $A$ be a semisimple $G$-algebra such that $A^G=k$. Then there is an imprimitive system $\{e_1,\ldots,e_n\}$ in $A$, such that $G$ acts transitively on this system, and $e_iA$ is a simple algebra, for all $i$.
\end{lem}

\begin{proof}
Applying Wedderburn's Theorem for semisimple algebras (see \cite[Theorem 4, pag 157]{Lorenz}) we conclude that there is a set $\{e_1,\cdots,e_n\}$ of primitive orthogonal central idempotents in $A$ such that $1_A= e_1+\cdots + e_n$, and  $e_iA$ is a simple algebra. It is immediate that $\{e_i\}$ is an imprimitive system in $A$. Suppose that $G$ does not act transitivity on $\{e_i\}$. Renumbering if is necessary, we can suppose that the orbit of $e_1$ is $\{e_1,\ldots, e_t\}$ and the orbit of $e_{t+1}$ is $\{e_{t+1},\ldots, e_r\}$, where $1\leq t<r\leq n$. Then $x_1=e_1+\cdots+e_t$ and $x_2=e_{t+1}+\cdots+e_{r}$ are linearly independent elements in $A^G$. However, this contradicts the fact that $\dim_kA^G=1$.
\end{proof}

\begin{obs}\label{obs galois complejo}
Following with the same hypothesis of the Lemma \ref{lema E_i sistema imprimitivo}, and supposing  $k=\mathbb C$, we have that $e_1A=\End_k(V)$ for some complex vector space $V$. Using the Skolem-Noether theorem, it is easy to see that every $S$-action over $\End_k(V)$, is defined by a projective representation $\rho:S\to \text{PLG}(V)$, $$(g\cdot f)(v) =\rho(g)( f (\rho(g)^{-1}(v))),$$ for $f\in \End_k(V), g\in G, v\in V$.
Note that  $\End_k(V)^S=k$ if and only if  $V$ is an irreducible projective representation.
\end{obs}

The following proposition is the Proposition 3.2 in \cite{Davy}, and we give an alternative proof.

\begin{prop}\label{prop inducida objeto de g}
Let $S$ be a subgroup of $G$, and $(B,\cdot)$ be an $S$-algebra. Then $B$ is a $k^S$-Galois object if and only if the induced algebra $\Ind_S^G(B)$ is a $k^G$-Galois object.
\end{prop}

\begin{proof}
We will use the Proposition \ref{caracterizacion A.Ga=ideales G-inv} for the proof.

For the Lemma \ref{lema inducida=B y algebra invariantes} and Proposition \ref{caracterizacion A.Ga=ideales G-inv}, it is clear that $B^S=k$ if and only if $\Ind_S^G(B)^G=k$. Moreover,  $\dim(B)=|S|$ if and only if $\dim(\Ind_S^G(B))=|G|$.

Let $\{e_1,\ldots,e_n\}$ be the canonical imprimitive system in $\Ind_S^G(B)$, see Remark \ref{nota sistema imprimitivo canonico}.

Suppose that $B$ is a $k^S$-Galois object. Let $J$ be a $G$-invariant left ideal in $\Ind_S^G(B)$. Then $e_1J$ is an $S$-invariant left ideal of $e_1\Ind_S^G(B)\simeq B$. Since,  $g_i\accion (e_1J)=e_iJ$ then $J=\bigoplus_{i=1}^ne_iJ$ is a trivial ideal.

Conversely, suppose that $\Ind_S^G(B)$ is a $k^G$-Galois object. Let $J$ be an $S$-invariant left ideal in $B\simeq e_1\Ind_S^G(B)$, since $\tilde{J}=g_1J\oplus\cdots \oplus g_nJ$ is a $G$-invariant left ideal in $\Ind_S^G(B)$. Then $J$ is a trivial ideal.
\end{proof}

\begin{teor}\label{teor principal cap 2}
Every $k^G$-Galois object is isomorphic as a $G$-algebra to the induced algebra $\Ind_S^G(B)$, where $B$ is a simple $k^S$-Galois object and $S$ is a subgroup of $G$.
\end{teor}

\begin{proof}
Suppose that $A$ is a $k^G$-Galois object. By Corollary  \ref{algebra de galois semisimple} the algebra $A$ is semisimple, so  by Lemma \ref{lema E_i sistema imprimitivo} and Proposition \ref{caracterizacion algebras imprimitivas} the $k^G$-Galois object $A$ is the induced of a simple Galois object. Thus the theorem follows by Proposition \ref{prop inducida objeto de g}.
\end{proof}

\begin{obs}
Following the Remark \ref{obs galois complejo}, by Proposition \ref{caracterizacion A.Ga=ideales G-inv} if $V$ is an irreducible  projective representation of the group $S$, the $S$-algebra $\End_{\mathbb C}(V)$, is a $\mathbb C^S$-Galois object if and only $\dim_{\mathbb C}(V)=|S|^{1/2}$. The condition $\dim_{\mathbb C}(V)=|S|^{1/2}$ implies that the 2-cocycle of $S$, associated to the projective representation is non-degenerated, see Definition \ref{defin cociclo non-dege}. Hence, in the complex case, the Galois objects of a group $G$ are classified by pairs $(S,\alpha)$, where $S\subseteq G$ is a subgroup, and $\alpha\in Z^2(S,\mathbb C^*)$ is a non-degenerated 2-cocycle, it is the main result of \cite{Mov}.
\end{obs}
\section{Simple Galois objects}\label{seccion 4}

In this section we show a structure theorem for simple $k^G$-Galois objects. We use the Miyashita-Ulbrich action in order to prove that every simple Galois object is isomorphic as a graded algebra to a twisted group algebra of a normal subgroup of $G$.

\subsection{Simple Galois objects as twisted group algebras}
Recall that a $k^G$-module algebra $B$ is the same as a $G$-graded algebra $B=\bigoplus_{g\in G}B_g$, where $B_g= B\triangleleft e_g$. Thus, if $A$ is a $k^G$-Galois object, the Miyashita-Ulbrich action of $k^G$ over $A$ defines a structure of $G$-graded algebra $A=\bigoplus_{g\in G}A_g$, where \[
A_{g}= \{a\in A| ab= (g\accion b)a \ \ \forall b\in A\}.
\]

\begin{prop}\label{proposicion propiedades graduacion}
Let $A$ be a $k^G$-Galois object, then

\begin{enumerate}[i)]
  \item $A_gA=AA_g$ is bilateral ideal of $A$,
  \item $A_g$ is a $\mathcal{Z}(A)$-submodule of $A$,
  \item $A_e=\mathcal{Z}(A)$,
  \item $g\accion A_x= A_{gxg^{-1}}$.
\end{enumerate}
\end{prop}
\begin{proof}
Straightforward.
\end{proof}

\emph{From now on we will denote by $A$ a $k^G$-Galois object, where $A$ is a central simple algebra over $K$.}

Let considerer the normal subgroup of $G$, given by \begin{equation}\label{ecuacion definicion subgrupo N}
N=\{x\in G|x\accion\alpha=\alpha,\ \ \forall\alpha\in K\}.
\end{equation}

\begin{lem}\label{lema equivalencias h en N}
The following affirmations are equivalent:
\begin{enumerate}[$i)$]
  \item  $x\in N$.
  \item $A_x\neq 0$.
  \item $A_gA_x=A_{gx}$ for all $g\in G$.
  \item $A_xA_{x^{-1}}=A_{x^{-1}}A_x=K$.
  \item There is an invertible element in $A_x.$
\end{enumerate}
\end{lem}

\begin{proof}
$i)\Rightarrow ii)$ Since $A$ is simple, by \cite[Example 2.4]{Takeuchi87} we have a category equivalence $$A\otimes_K(-):\ _K\mathcal{M}\to\ _A\mathcal{M}_A,\ \ V\mapsto A\otimes V,$$with quasi-inverse functor given by $$(-)^A:\ _A\mathcal{M}_A\to\ _K\mathcal{M}, \  \  M\mapsto M^A=\{m\in M| am=ma, \forall a\in A\}.$$   For every $x\in N$, we define an $A$-bimodule $A^x$, where the right action is the multiplication and left action is given by $b\cdot a= (x\rightharpoonup b) a$ for all $b\in A, a\in A^x$. Thus, $(A^x)^A= A_x= \{a\in A| ab= (x\accion b)a, \ \ \forall b\in A\}$, and since $A^x\neq 0$, then $A_x\neq0$.\\

$ii)\Rightarrow iii)$  Recall that $A_xA=AA_x$ is a bilateral ideal. Since $A$ is simple, $0\varsubsetneq A_x\sset A_xA=A$. The equation
  $$\bigoplus_{g\in G}A_g=A=A_xA=\bigoplus_{g\in G}A_xA_g,$$ implies $A_xA_{g}=A_{xg}$.\\

$iii) \Rightarrow iv) \Rightarrow v)$ Straightforward.\\

$v)\Rightarrow i)$ Let $u_x\in A_x$ be a unit, and let $\alpha\in K$. By the characterization of Miyashita-Ulbrich, we have
$$[(x\accion \alpha)-\alpha]u_x=0,$$
hence,  $(x\accion \alpha)-\alpha=0$.
\end{proof}

\begin{obs}\label{obs base twisted}
Let $S$ be a finite group,  $K$ a field, and $\s\in Z^2(S, K^*)$ a 2-cocycle. For each $s\in S$, we will use the notation $u_s \in K_{\s}S$ to indicate the corresponding element in the \emph{twisted group algebra} $K_{\s}S$. Thus $(u_s)_{s\in S}$ is a $K$-basis of $K_{\s}S$, and in this basis $u_su_t = \s(s, t)u_{st}$.
\end{obs}

Recall that an element $s\in S$ is called $\s$-\emph{regular} if $\s(s, t) = \s(t, s)$ for all $t \in C_S(s)$. This definition depends only on the class of $s$ under conjugation.
\begin{defin}\label{defin cociclo non-dege}
Let $\sigma\in Z^2(S,K^*)$ be a 2-cocycle. The 2-cocycle $\s$ is called \emph{non-degenerate} if and only if $\{1\}$ is the only $\s$-regular class in $S$.

\end{defin}

\begin{obs}\label{remark non-degenerate}
The 2-cocycle $\s$ is non-degenerate if and only if $\dim_K \mathcal{Z}(K_\s S)=1$  (see \cite[Theorem 9.3, pag 410]{karpi}).
\end{obs}

\begin{prop}\label{prop A simeq K_sigma N como algebras}
Let $A$ be a $k^G$-Galois object, where $A$ is a central simple algebra over $K$. Then $A$ is isomorphic to a twisted group algebra $K_\sigma N,$ where $N$ is the normal subgroup of  $G$ that stabilizes $K$, and $\sigma:N\times N\to K^*$ is a non-degenerated 2-cocycle. Moreover, if we suppose $A=K_\sigma N$, then the action of $x\in N$ over $a\in A$ is given by $$x\accion a= u_xau_x^{-1}.$$
\end{prop}
\begin{proof}
By the Lemma \ref{lema equivalencias h en N}, $A$ is a crossed ring over $N$. Since $A_e=\mathcal{Z}(A)=K$, $A$ is isomorphic to a twisted group algebra $K_\s N$, thus by Remark \ref{remark non-degenerate} $\s$ is a non-degenerate 2-cocycle.

Finally, if $x\in N$, then
\begin{align*}
    y\in A_x=Ku_x \ \leftrightarrow \ \ ya=(x\accion a)y \ \ \forall a \in A.
\end{align*}
Taking  $y=u_x$, we have $x\accion a=u_xau_{x}^{-1}$ for all $a\in A$.
\end{proof}

\subsection{The function $\g$}

The group $G$ acts over $K$ through the natural projection,
\begin{equation}\label{ecuacion proyeccion natural}
G\to G/N\to \Gal(K|k),\ \ g\mapsto\overline{g}.
\end{equation}

By item \emph{(iv)} in Proposition \ref{proposicion propiedades graduacion}, the action of $G$ over $A$ defines a function $\gamma:G\times N\to K^*$ determinate by the equation
\begin{equation}\label{ecua definicion gamma}
  g\accion u_x= \gamma(g,x)u_{\-^gx} \quad\quad (\-^gx:=gxg^{-1}).
\end{equation}

\begin{prop}
For all $g, h\in G$ and $x, y\in N,$ the function $\gamma:G\times N\to K^*$ defined in \eqref{ecua definicion gamma} satisfies the following equations:
\begin{align}
\intertext{ Condition C1:}
\gamma(x,y)\s(x,x^{-1})&=\s(x,y)\s(xy,x^{-1}).\label{condicion 1 gamma} \tag{C1}\\
\intertext{ Condition C2:} \bar{g}(\sigma(x,y))\gamma(g,xy)&=\sigma(\-^gx,\-^gy)\gamma(g,x)\gamma(g,y).\label{condicion 2 gamma} \tag{C2}\\
\intertext{ Condition C3:}
\gamma(gh,x)&=\bar{g}(\gamma(h,x))\gamma(g,\- ^hx).\label{condicion 3 gamma}\tag{C3}
\end{align}
\end{prop}

\begin{proof}
The conditions \eqref{condicion 1 gamma}, \eqref{condicion 2 gamma}, and \eqref{condicion 3 gamma}  are equivalent to
\begin{align*}
    x\accion  u_y&=u_x u_yu_x^{-1},\\
    g\accion  u_xu_y&=(g\accion u_x)(g\accion  u_y),\\
    g\accion (h \accion u_x)&=gh\accion u_x,
\end{align*}respectively.
\end{proof}

\begin{nota}\label{Remark N abeliano}
  If $N\subseteq \mathcal{Z}(G)$, the condition \eqref{condicion 1 gamma} is equivalent to
\begin{equation*}\label{condicion C1 Abeliano}
\gamma(x,y)=\frac{\s(x,y)}{\s(y,x)}=:\mathrm{Alt}_\s(x,y).
\end{equation*}
Moreover, if $\sigma (N,N),\g(G,N)\subseteq k^*$, the conditions \eqref{condicion 2 gamma} and \eqref{condicion 3 gamma} are equivalent to
\begin{align}
\gamma(g,xy)&=\gamma(g,x)\gamma(g,y)\label{condicion C2 valores en k}\\
\gamma(gh,x)&=\gamma(g,x)\gamma(h,x)\label{condicion C3 valores en k}
\end{align}
for all $g,h\in G$ and $x,y\in N$. A function $\g$ that satisfies \eqref{condicion C2 valores en k} and \eqref{condicion C3 valores en k} is called a \emph{pairing}.
\end{nota}

\section{Construction of $k^G$-Galois objects from Group-Theoretical data}\label{seccion 5}

In this section we define group-theoretical data associated to a finite group $G$ and a field $k$. We define a $k^G$-Galois object for each associate datum  to $G$ and $k$, and we establish an equivalence among the data, such that equivalent data define isomorphic Galois objects.

\subsection{Data associated with Galois objects}\label{subseccion datos equivalentes}
Let $G$ be a finite group, and let $k$ be a field.

\begin{defin}
A\emph{Galois datum associated to $k^G$} is a collection $(S,K,N,\s,\gamma)$ such that
\begin{enumerate}[$i)$]
\item $S$ is a subgroup of $G$ and $N$ is a normal subgroup of $S$.
\item $K\supseteq k$ is a Galois extension with Galois group $S/N$.
\item char$(k)\nshortmid |N|$.
\item $\s:N\times N\to K^*$ is a non-degenerate 2-cocycle.
\item $\gamma:S\times N\to K^*$  satisfies the equations \eqref{condicion 1 gamma}, \eqref{condicion 2 gamma}, and \eqref{condicion 3 gamma}.
\end{enumerate}
\end{defin}

Let $(S, K,N,\s,\gamma)$ be a Galois datum associated to $k^G$. We will denote by $A(K_\sigma N,\gamma)$ the twisted group algebra $K_\sigma N$  with $S$-action defined by
\begin{equation*}%\label{ecuacion definicion gamma}
g\accion \alpha u_x=(g\accion \alpha)(g\accion u_x)=\bar{g}(\alpha)\gamma(g,x)u_{\-^gx},
\end{equation*}
for $g\in S$, $x\in N$, and $\alpha\in K$.

We will denote by $\Ind_S^G(A(K_\sigma N,\gamma))$ the induced $G$-algebra from the $S$-algebra $A(K_\sigma N,\gamma)$.

\begin{prop}\label{prop dato implica algebra de galois simple}
Let $(S,K,N,\s,\gamma)$ be a Galois datum associated to $k^G$. The $S$-algebra $A(K_\sigma N,\gamma)$ is a simple $k^S$-Galois object, with $\mathcal{Z}(A(K_\sigma N,\gamma))=K$.
\end{prop}

\begin{proof}%[Proof of Proposition \ref{prop dato implica algebra de galois simple}]

Since char$(k)\nshortmid |N|$, by Maschke's Theorem for twisted group algebra (see \cite[2.10 Theorem]{karpi}), the algebra $K_\sigma N$ is a semisimple algebra. Moreover, since $\sigma$ is non-degenerate then $ \mathcal{Z}(K_\sigma N)=K$. Hence $K_\sigma N$ is a central simple algebra over $K$.

Let us denote  by $A$ the $S$-algebra $A(K_\sigma N,\gamma)$. If we see that $A$ is a $K^N$-Galois object, then $A$ is a $k^S$-Galois object by Corollary  \ref{prop teorema takeuchi importante}.

Recall, if $A$ is a central simple algebra over $K$, the map
\begin{align*}
    A\otimes_K A^{op}&\to \End_K(A)\\
    a\otimes b&\mapsto [t\mapsto atb]
\end{align*}
is an isomorphism (see \cite[pag. 32]{GS}). The map $\theta_N:A\# N\to \End_K(A)$,  $\theta_N(a\#x)(b)=a(x\accion b)$ is surjective, since that  $$\theta_N(u_xu_y\#y^{-1})(a) = u_xau_y.$$ Now, $\dim_K(A\# N)= \dim_K (\End_K(A))$, hence $\theta_N$ is bijective. Thus,  $A$ is an $K^N$-Galois object by Proposition \ref{prop equivalencia galois}.
\end{proof}

\subsection{Equivalence of Galois data}

\begin{defin}
We will say that $(S,K,N,\s,\gamma)$ and $(S',K',N',\s',\gamma')$ are \emph{equivalent Galois data associated to $k^G$} if there exists a $G$-algebra isomorphism between $\Ind_S^G(A(K_\sigma N,\gamma))$ and $\Ind_{S'}^G(A(K'_{\sigma'} N',\gamma')).$

% If $G=S$, it is clear that $\Ind_S^G(A(K_\sigma N,\gamma))=A(K_\sigma N,\gamma)$. Then $(K,N,\s,\gamma)$ and $(K',N',\s',\gamma')$ are \emph{equivalent simple data associates to $k^G$} if there exists a $G$-algebra isomorphism between $A(K_\sigma N,\gamma)$ and $A(K'_{\sigma'} N',\gamma').$
\end{defin}
\begin{defin}
Let $G$ be a finite group, $S$ a subgroup of $G$ and $(A,\cdot)$ an $S$-algebra. For each $g\in G$, we considerer the $g^{-1}Sg$-algebra $(A^{(g)},\cdot_g)$, such that $A^{(g)}=A$ as algebras and $g^{-1}Sg$-action give by $$h\cdot_g a=(g hg^{-1})\cdot a,$$
for all $h\in g^{-1}Sg$ and $a\in A^{(g)}$.
\end{defin}

\begin{lem}\label{lema isomorfismo Ind_S^G(A) and Ind_g(-1)Sg^G(A^(g))}
  Let $G$ be a finite group and $S$ be a subgroup of $G$. For all $g\in G$, there exists a $G$-algebra isomorphism between $\Ind_{S}^G{(A)}$ and $\Ind_{g^{-1}Sg}^G(A^{(g)})$.
\end{lem}

\begin{proof}
 For each $g\in G$, considerer the map
\begin{align*}
  \psi_g:\Ind_S^G(A)&\to \Ind_{g^{-1}Sg}^G(A^{(g)})\\
f&\mapsto \psi_g(f)=[h\mapsto f(gh)].
\end{align*}
It is easy to check that $\psi_g$ is a $G$-algebra isomorphism.
\end{proof}

Recall that two transitive $G$-sets $G/S$ and $G/S'$ are isomorphic if and only if there exists $g\in G$ such that $S'=g^{-1}Sg$.

\begin{lem}\label{lema inducidas isomorfas}
  Let $S, S'$ be two subgroups of $G$, let $A$ be a simple $S$-algebra, and $B$ be a simple $S'$-algebra. Then there exists a $G$-algebra isomorphism between $\Ind_S^G(A)$ and $\Ind_{S'}^G(B)$ if and only if $S=g^{-1}S'g$ and $A\simeq B^{(g)}$ as $S$-algebras.
\end{lem}

\begin{proof}
  Let $\{e_i\}_{1\leq i\leq[G:S]}$, $\{r_i\}_{1\leq i\leq[G:S']}$ and $\{e'_i\}_{1\leq i\leq[G:g^{-1}S'g]}$ be the canonical imprimitive systems associated to $\Ind_S^G(A)$, $\Ind_{S'}^G(B)$ and $\Ind_{g^{-1}S'g}^G(B^{(g)})$ (see Remark \ref{nota sistema imprimitivo canonico}).

Suppose that $\chi:\Ind_S^G(A)\to \Ind_{S'}^G(B)$ is a $G$-algebra isomorphism. Since that $A$ and $B$ are simple algebras, the elements in $\{e_i\}_{1\leq i\leq[G:S]}$, and $\{r_i\}_{1\leq i\leq[G:S']}$ are the central primitive idempotents for $\Ind_S^G(A)$, and $\Ind_{S'}^G(B)$, respectively. Then $\chi$  induce a $G$-set isomorphism between $\{e_i\}\simeq G/S$ and $\{r_i\}\simeq G/S'$. Hence, there exist $g\in G$ such that $S=g^{-1}S'g$.

Now, by the Lemma \ref{lema isomorfismo Ind_S^G(A) and Ind_g(-1)Sg^G(A^(g))} we have the following $G$-algebra isomorphism
$$\Ind_S^G(A)\stackrel{\chi}{\longrightarrow}\Ind_{S'}^G(B)
\stackrel{\psi_g}{\longrightarrow}\Ind_{g^{-1}S'g}^G(B^{(g)}), $$
then we have that
$$A\simeq e_1\Ind_{S}^G(A)\stackrel{\psi_g\circ \chi}{\longleftrightarrow} e'_1\Ind_{S}^G(B^{(g)})\simeq B^{(g)}. $$

Conversely, suppose that $S=g^{-1}S'g$ and $A\simeq B^{(g)}$ as $S$-algebras. We can construct a $G$-algebra isomorphism between $\Ind_S^G(A)$ and $\Ind_{S}^G(B^{(g)})$, using the Remark \ref{nota ind is a functor}  and the Lemma \ref{lema isomorfismo Ind_S^G(A) and Ind_g(-1)Sg^G(A^(g))}.
\end{proof}

Let $f:A\to A'$ be an $S$-algebra isomorphism between  $A=A(K_\sigma N,\gamma)$ and $A'=A(K'_{\sigma'} N',\gamma')$. By the definition of the Miyashita-Ulbrich action, $f$ is an $S$-graded algebra isomorphism. Since  $N=\{x\in S| A_x\neq 0\}$ (see Lemma \ref{lema equivalencias h en N}), then $N=N'$. Moreover,  $f|_{K}$ is a field isomorphism, so we can suppose without loss of generality that  $K=K'$.

\begin{prop}\label{prop equivalencia simple}
The Galois data $(S,K,N,\sigma,\gamma)$ and $(S,K,N,\s',\gamma')$ associated to $k^S$ are equivalent if and only if there are $\omega\in \mathcal{Z}(\Gal(K|k))$, and $\eta:N\to K^*$, such that
\begin{align}
\omega(\sigma(x,y))\eta(xy) &= \eta(x)\eta(y)\sigma'(x,y) \label{condiciones datos equivalentes 1}\\
    \omega(\gamma(s,x))\eta(\-^sx) &= \bar{s}(\eta(x))\gamma'(s,x),\label{condiciones datos equivalentes 2},
\end{align}for all $x,y\in N$, $s\in S$.
\end{prop}

\begin{proof}
Let $\{u_x\}$ and $\{u'_x\}$ be a $K$-basis of $A$ and $A'$, see Remark \ref{obs base twisted}.

Suppose that $f:A(K_\sigma N, \gamma)\to A(K_{\sigma'}N,\gamma')$ be an $S$-algebra isomorphism. Considerer $\omega :=f|_{K}\in \mathcal{Z}(\text{Gal}(K|k))$, and $\eta:N\to K^*$ determined by the equation $$f(u_x)=\eta(x)u'_x\quad (\forall x\in N).$$
A straightforward computation shows that $\omega$ and $\eta$ satisfy the equations \eqref{condiciones datos equivalentes 1}, and \eqref{condiciones datos equivalentes 2}.

Conversely, it is easy to check that if $\omega$ and $\eta$ satisfy the equations \eqref{condiciones datos equivalentes 1}, \eqref{condiciones datos equivalentes 2}, then $f(\alpha u_x) :=\omega(\alpha)\eta(x)u'_x$ is an $S$-algebra isomorphism.
\end{proof}
\begin{lem}\label{lema B(g)}
Let $(S,K,N,\s,\g)$ be a Galois datum associated to $k^G$. For every $g\in G$, the Galois datum associated to $A(K_{\s}N,\g)^{(g)}$ is given by $(g^{-1}Sg,K,N,\s^{(g)},\g^{(g)})$, where $\s^{(g)}(x,y)=\s(gxg^{-1},gyg^{-1})$, and $\g^{(g)}(h,x)=\gamma(ghg^{-1},gxg^{-1})$, for all $x, y \in N$, $h\in g^{-1}Sg$.
\end{lem}

\begin{proof}
Let $A=A(K_{\s}N,\g)$. Note that $(A^{(g)})_x= A_{gxg^{-1}}$, then $\{u_{gxg^{-1}}| x\in N\}$ is a basis of $A^{(g)}$, where $u_{gxg^{-1}}\in (A^{(g)})_x$. Hence, $\s^{(g)}(x,y)=\s(gxg^{-1},gyg^{-1})$, and $\g^{(g)}(h,x)=\gamma(ghg^{-1},gxg^{-1})$, for all $x, y \in N$, $h\in g^{-1}Sg$.
\end{proof}

\begin{teor}
  The data $(S,K,N,\s,\g)$ and $(S',K,N,\s',\g')$ are equivalent if and only if there exist $g\in G$, $\omega\in \mathcal{Z}(\Gal(K|k))$ and $\eta:N\to K^*$ such that $S=g^{-1}S'g$ and
\begin{align*}
\omega(\sigma(x,y))\eta(xy) &= \eta(x)\eta(y)\sigma'(gxg^{-1},gyg^{-1}),\\
    \omega(\gamma(s,x))\eta(\-^sx) &= \bar{s}(\eta(x))\gamma'(gsg^{-1},gxg^{-1}),
\end{align*}
for all $x,y\in N$, $s\in S$.
\end{teor}
\begin{proof}
Since $A(K_{\s}N,\g)$ and $A(K_{\s'}N,\g')$ are simple algebras, the proof of theorem follows by Lemma \ref{lema B(g)}, Lemma \ref{lema inducidas isomorfas}, and Proposition \ref{prop equivalencia simple}.
\end{proof}

\section{Classification of $k^G$-Galois objects}\label{seccion 6}

In this section we aim to give a proof of our main result, \textit{i.e.}, the Theorem \ref{main result}.

\begin{lem}\label{lema K_sigma N galois entoces char no divide orden}
Let $(S,K,N,\s,\gamma)$ be a Galois datum associated to $k^S$. If $A(K_\sigma N,\gamma)$ is a $k^S$-Galois object then char$(K)\nshortmid |N|$. In particular, char$(k)\nshortmid |N|.$
\end{lem}

\begin{proof}
By the Corollary \ref{prop teorema takeuchi importante}, the algebra $K_\sigma N$ is a $K^N$-Galois object, so the map $\theta:K_\sigma N\sedi N\to \End(K_\sigma N)$ is an isomorphism. Note that $$\theta(\sum_{y\in N}u_e\# y)(u_x)=\sum_{y\in N}u_yu_xu_y^{-1}$$ lies in the center of the twisted group algebra. Then $\theta(\sum_{y\in N}u_e\# y)(u_x)=0$ if $x\neq e$ and $\theta(\sum_{y\in N}u_e\# y)(u_e)=|N|u_e$. Hence, $|N|\neq 0.$
\end{proof}

\begin{teor}\label{Teorema casiPrincipal}
Let $S$ be a finite group, $k$ be  a field, and $A$ be a simple $S$-algebra. Then $A$ is a $k^S$-Galois object if and only if there exists $(S,K,N,\s,\gamma)$ a Galois datum associated to $k^S$,  such that
$$A\simeq A(K_\sigma N,\gamma),$$ as $S$-algebras.
\end{teor}

\begin{proof}
Suppose that $A$ is a $k^S$-Galois object. Let $K=\mathcal{Z}(A)$ and
$N=\{g\in S|g\accion \alpha=\alpha \ \ \forall \alpha\in K\}$.  $A$ is isomorphic to the twisted group algebra $K_\sigma N$, where $\sigma$ is non-degenerate (see Proposition \ref{prop A simeq K_sigma N como algebras}). Now, char$(k)\nshortmid |N|$ by Lemma \ref{lema K_sigma N galois entoces char no divide orden}. The function $\gamma$ defined by the equation $$g\accion \alpha u_x=\bar{g}(\alpha)\gamma(g,x)u_{\-^gx},\quad \  \  \  (g\in S, x\in N, \alpha\in K),$$ completes a Galois datum $(S,K,N,\s,\gamma)$, such that $A\simeq A(K_\sigma N,\gamma)$.
\end{proof}

Now our main result follows immediately from our previous results.

\begin{proof}[Proof of Theorem \ref{main result}]
It follows from Theorem \ref{Teorema casiPrincipal}, and Theorem \ref{teor principal cap 2}.
\end{proof}

\subsection{Two families of Examples}

Let $G=\mathbb{Z}_n\oplus\mathbb{Z}_n\oplus\mathbb{Z}_n$ and $N=\{0\}\oplus \mathbb{Z}_n\oplus\mathbb{Z}_n\subseteq G$ with $n$ an odd prime. Considerer the fields $k=\mathbb{Q}[\zeta]$ and $K=\mathbb{Q}[q]$, where $q=e^{\frac{2\pi i}{n^2}}$ is a primitive $n$th root of unity and $\zeta= q^n$. Now, define a 2-cocycle $\s:N\times N\to k^*$ of $N$ with values in $\langle\zeta\rangle^*\subseteq k^*$, by $$\s(x,y)=\zeta^{x_2y_1-x_1y_2}\quad (x=(x_1,x_2), y=(y_1,y_2)\in N).$$ Since
$\{0\}=\{x\in N|\s(x,y)=\s(y,x)\ \forall y\in N\},$
thus $\s$ is a non-degenerate 2-cocycle.

In this case, by the Remark \ref{Remark N abeliano} a function $\g:G\times N\to k^*\subseteq K^*$ satisfies the conditions \eqref{condicion 1 gamma}, \eqref{condicion 2 gamma} and \eqref{condicion 3 gamma} if and only if $\g$ is a pairing such that $\g|_{N\times N}=\mathrm{Alt}_\s=\s^2$. Hence, it is easy to see that the collection $(G,K,N,\s,\g)$ is a Galois datum associated to $k^G$. Moreover, if $\g'$ is another paring such that $\g'|_{N\times N}=\s^2$, then the Galois data $(G,K,N,\s,\g)$ and $(G,K,N,\s,\g')$ are equivalent if and only if $\g=\g'$.

We can obtain other family of Galois data taking $G$, $N$, $k$, and $K$ as above, and defining the 2-cocycle $\s':N\times N\to k^*$ by
$$\s'(x,y)=\zeta^{x_2y_1}\quad (x=(x_1,x_2), y=(y_1,y_2)\in N),$$
and  a pairing $\g:G\times N\to k^*$, such that
$$\g(x,y)=\mathrm{Alt}_{\s'}(x,y)=\zeta^{x_2y_1-x_1y_2},$$
for all $x,y\in N$.

\section{Obstruction theory for Galois data}\label{seccion obstruccion}
The main result of this section is to present some necessaries conditions for the existence of a function $\gamma:G\times N\to K^*$, such that $(S,N,K,\sigma, \gamma)$ is a Galois datum associated to $k^G$.
\subsection{Hochschild cohomology for groups}

We briefly recall the well-known Hochschild cohomology for the algebra $\mathbb ZG$.

Let $G$ be a group and $A$ be a $G$-bimodule. Define $C^0(G,A)=A$, and for $n\geq 1$ $$C^n(G,A)=\{f:\underbrace{G\times\cdots \times G}_{n-times}\to A| f(x_1\ldots,x_n)=0, \text{ if } x_i=1_G \text{ for some }i \}.$$

Considerer the following cochain complex
\begin{equation*}\label{complex}
0 \longrightarrow C^0 (G, A) \stackrel{d_0}{\longrightarrow }
C^1 (G, A) \stackrel{d_1}{\longrightarrow }C^2 (G, A) \cdots C^{n} (G, A)
\stackrel{d_n}{\longrightarrow } C^{n+1} (G, A) \cdots
\end{equation*} where
\begin{align*}
    d_n(f)(x_1,x_2,\ldots,x_{n+1})&=x_1\accion f(x_2,\ldots,x_{n+1})\\
    &+\sum_{i=1}^n(-1)^{i}f(x_1,\ldots,x_{i-1},x_ix_{i+1},x_{i+2},\ldots,x_{n+1})\\
    &+(-1)^{n+1}f(x_1,\ldots,x_{n})\acciond x_{n+1}.
\end{align*}
Then as usual define $HZ^n(G,A):=\ker(d_n)$, $HB^n(G,A):= \text{Im}(d_{n-1})$ and $HH^n(G,A):=HZ^n(G,A)/HB^n(G,A)$ ($n\geq 0$) the Hochschild cohomology of $G$ with coefficients in $A$.

\begin{obs}
A left $G$-module $A$, is a $G$-bimodule with the trivial right action, in this case the differential maps will be denoted by $\delta_n$, and the Hochschild Cohomology of $G$ with coefficients in $A$, is the ordinary group cohomology.
\end{obs}
Let $G$ be a finite group, $N$ a normal subgroup of $G$, $K$ a Galois extension of $k$
with Galois group $G/N$, and $\sigma\in Z^2(N,K^*)$ a non-degenerate 2-cocycle.

Thus the abelian group $$C^1(N,K^*)=\{f:N\to K^*|f(e)=1\},$$
is a $G$-bimodule with left action
\begin{equation}\label{ecuacion accion de C^1 a izquierda}
    (g\accion f)(x)=\bar{g}(f(x)),
\end{equation}
and right action
\begin{equation}\label{ecuacion accion de C^1 a derecha}
    (f\acciond g)(x)=f(\-^gx)   \  \  \  (\-^gx=gxg^{-1}),
\end{equation}
for all $g\in G, x\in N$, and  $f\in C^1(N,K^*)$.

By abuse of notation we will identify to a function $\gamma: G\to C^1(N,K^*)$ with its associated function $\gamma: G\times N\to K^*,$ and vice versa.

\subsection{Obstructions}

The abelian groups $C^n(N,K^*), Z^n(G,K^*)$ are $G$-bimodules with left action $(g\cdot \sigma)(x_1,\ldots,x_n)=\overline{g}(\sigma(x_1,\ldots,x_n))$, and right action $(\sigma^g)(x_1,\ldots,x_n)= \sigma(^gx_1,\ldots,^gx_n)$. Analogously, the abelian group $\Hom(N,K^*)=\widehat{N}$ is a $G$-bimodule.

\begin{prop}[First obstruction]
There exists a function $\gamma:G\times N\to K^*$ that satisfied \eqref{condicion 2 gamma}
if and only if the second cohomology class of $\frac{g\cdot\sigma}{\sigma^g}$ is zero for all $g\in G$.
\end{prop}
\begin{proof}
It follows immediately by the condition   \eqref{condicion 2 gamma}.
\end{proof}

%\subsubsection{Second obstruction}
Suppose that the first obstruction is zero, then for all $g\in G$ there exists $\gamma_g:N\to K^*$, such that $$\delta_1(\gamma_g)=\frac{g\cdot\sigma}{\sigma^g}\in B^2(N,K^*),$$ it defines a function $\gamma: G\to C^1(N,K^*)$, $g\mapsto \g_g$.

\begin{lem}
For all $g,h\in G$, $d_1(\gamma)(g,h)\in \widehat{N}$, the function $d_1(\gamma):G\times G\to \widehat{N}$ is a Hochschild 2-cocycle of $G$ with values in $\widehat{N}$.
\end{lem}
\begin{proof}
If $g,h\in G$, then
\begin{align*}
\delta_1(d_1(\gamma)(g,h))&=\delta_1[(g\cdot \gamma_{h})(\gamma_{gh})^{-1}(\gamma_{g}^{h})]\\
&=\left(g\cdot \frac{h\cdot \sigma}{\sigma^{h}}\right)\left(\frac{\sigma^{gh}}{gh\cdot \sigma}\right)\left(\frac{g\cdot \sigma^h}{(\sigma^g)^h}\right)\\
&=1.
\end{align*}
\end{proof}

A straightforward computation shows that if $\gamma':G \to C^1(N,K^*)$, $g\mapsto\g'_g$, is another function such that $\delta_1(\gamma'_g)=\frac{g\cdot\sigma}{\sigma^g}$ for all $g\in G$, then the Hochschild 2-cocycles $d_1(\gamma')$ and $d_1(\gamma)$ are cohomologous.

\begin{prop}[Second obstruction]
 There exists a function $\gamma:G\times N\to K^*$ that satisfied \eqref{condicion 3 gamma} if and only if the second Hochschild cohomology class of $d_1(\gamma)$ is zero.
\end{prop}
\begin{proof}
If there is a function $\gamma: G\times N\to K^*$, that satisfies \eqref{condicion 3 gamma}, then $d_1(\gamma)=0$, so its cohomology is zero too.

Conversely, if $\gamma\in C^1(G, C^1(N,K^*))$ is a function, such that there exists  $\theta\in C^1(G, \widehat{N})$ with $d_1(\theta)= d_1(\gamma)$, then the function $\overline{\gamma}:G\times N\to K^*$ defined by $\overline{\gamma}(g,n)= \theta_g(n)^{-1}\gamma(g,n)$ for all $g\in G, n\in N$, satisfies  \eqref{condicion 3 gamma}.
\end{proof}

%\bibliographystyle{abbrv}
%\bibliography{bibliografiaMM}

\end{document}